\documentclass[12pt]{amsart}

\usepackage{amscd,amssymb,amsmath,latexsym}
\usepackage[mathscr]{euscript}
\usepackage{epsfig}

\newcommand{\N}{\mathbb N}

\newcommand{\R}{\mathbb R}
\newcommand{\K}{\mathbb K}

\newcommand{\cB}{\mathcal B}

\newtheorem{Thm}{Theorem}[section]
\newtheorem{Cor}[Thm]{Corollary}
\newtheorem{Lem}[Thm]{Lemma}
\newtheorem{Prop}[Thm]{Proposition}
\newtheorem{Rmk}[Thm]{Remark}

\newtheorem{Def}[Thm]{Definition}

\begin{document}
\title{A Connection Between Mixing and Kac's Chaos}

\author{George Androulakis and Rade Musulin}
\address{Department of Mathematics, University of South Carolina, 
Columbia, SC 29208}
\email{giorgis@math.sc.edu, musulin@math.sc.edu}



\thanks{The article is part of the second author's Ph.D. thesis which is prepared at the University of South Carolina under the supervision of the first author.}

\maketitle

\begin{abstract}
The Boltzmann equation is an integro-differential equation which 
 describes the density function of the distribution of the velocities of the molecules of dilute monoatomic gases under the assumption that the energy 
 is only transferred via collisions between
 the molecules. In 1956 Kac  studied the Boltzmann equation and defined a property of the density function that he called  the ``Boltzmann property" 
 which describes the behavior of the density function at a given fixed time as the number of particles tends to infinity.  The Boltzmann property 
 has been studied extensively since then, and now it is simply called chaos, or Kac's chaos. 
 On the other hand, in ergodic theory, chaos usually refers to the mixing properties of a dynamical system
 as time tends to infinity. A relationship is derived between Kac's chaos and the notion of mixing.
\end{abstract}

\section{Several notions of chaos}

The notion of ``chaos'' in ergodic theory, has its origins in the works of Poincare at the end of the 19th century.
Its meaning is dynamical randomness  of physical quantities that evolve with time. 
The set up for measure theoretic dynamical systems consists of a probability  space $(\Omega, \Sigma , \mu)$ which is called the 
{\em phase space}, 
and either a measurable map $S: \Omega \to \Omega$ (in the case of discrete time dynamical systems),
or a family of measurable maps $S_t: \Omega \to \Omega$ for $t \geq 0$ (in the case of continuous time dynamical systems)
satisfying $S_t \circ S_s = S_{t+s}$ for all $s,t \in [0,\infty)$ (semigroup property). Such tuple $(\Omega, \Sigma , \mu , S)$ or
$(\Omega , \Sigma , \mu , (S_t)_{t \geq 0})$ is called a {\em measure theoretic dynamical system}, or simply a {\em dynamical system}.
In the case of discrete time 
dynamical systems, the composition of the map $S$ with itself $n$ many times,  (where $n$ is a non-negative integer), 
is usually denoted as $S^n$, (a notation which resembles powers of $S$), and plays the role of $S_n$ that appears in the above definition of
continuous time dynamical systems. For simplicity we only consider discrete time dynamical systems and we keep in mind that 
the ``exponent" $n$ that appears in the compositions $S^n$ represents time. 
The maps $S^n$ can be thought to act on $\Omega$,
(by the formula $\Omega \ni \omega \mapsto S^n \omega$), or on real valued functions on $\Omega$,
(where the action of $S^n$ on such function $f$ produces the real valued function 
$\Omega \ni \omega \mapsto f(S^n\omega)$),
or on probability measures on $\Sigma$ (where the action of $S^n$ on such measure $\nu$ produces the measure
$\Sigma \ni A \mapsto \nu (S^{-n} A)$).
Thus one can study orbits of points of $\Omega$, (i.e. the 
sequence of points $(S^n (\omega))_{ n \in \N \cup \{ 0 \} }$), or orbits of real valued functions on $\Omega$,
or orbits of probability measures on $\Sigma$. The property of chaos in ergodic theory 
refers to the randomness of these orbits and it is explicitly quantified and studied in the books of ergodic theory.
An excellent book on this subject is the book of Arnold and Avez, \cite{ArnoldAvez}, or the short survey of 
Sinai \cite{Sinai}. Two quantifications of the notions of chaos in the measure theoretic ergodic theory are 
the notions of the ``stationary limit" and ``mixing":

\begin{Def} \label{Def:ergodicity}
\begin{itemize}
\item[(i)] We say that a dynamical system $(\Omega,\Sigma,\mu,S)$ is asymptotically stationary  with stationary limit $\nu$ if 
$\nu(A)=\lim\limits_{k \rightarrow \infty}\mu(S^{-k}A)$ for each $A \in \Sigma$.
\item[(ii)] We say that a dynamical system $(\Omega,\Sigma,\mu,S)$ is mixing if
$$
\lim\limits_{k \rightarrow \infty} \left| \mu(S^{-k}A \cap B) - \mu(S^{-k}A)\mu(B)\right|=0
\quad \text{for all }A,B \in \Sigma .
$$
\end{itemize}
\end{Def}

Note that if the members of a sequence of probability measures are defined on a common $\sigma$-algebra $\Sigma$ and converge at every fixed element of $\Sigma$
then the limit is also a probability measure \cite[Theorem $4.6.3(i)$]{Bogachev}. Thus the limit $\nu$ that is obtained in Definition~\ref{Def:ergodicity}(i)
is a probability measure, since obviously, for every $k \in \N$, the map $\Sigma \ni A \mapsto \mu (S^{-k}A)$ defines a probability measure on $\Sigma$.
Obviously, if a dynamical system $(\Omega,\Sigma,\mu,S)$ is asymptotically stationary with stationary limit $\nu$
then $\nu$ is {\em invariant under $S$}, (or equivalently, {\em $S$ is $\nu$-measure preserving}), i.e. 
\begin{equation} \label{measurepreserving}
\nu (S^{-1}(A)) = \nu (A) \quad \text{for all }A \in \Sigma .
\end{equation}
It is also obvious that if the dynamical system $(\Omega,\Sigma,\mu,S)$ is asymptotically stationary with stationary limit $\nu$
then it is mixing if and only if 
$$
\lim\limits_{k \rightarrow \infty} \left| \mu(S^{-k}A \cap B) - \nu(A)\mu(B)\right|=0
\quad \text{for all }A,B \in \Sigma .
$$
In particular, if $(\Omega,\Sigma,\mu,S)$ is a dynamical system and the map $S$ is $\mu$-measure preserving, then 
$(\Omega,\Sigma,\mu,S)$ is mixing  if and only if
$$
\lim\limits_{k \rightarrow \infty} \left| \mu(S^{-k}A \cap B) - \mu(A)\mu(B)\right|=0
\quad \text{for all }A,B \in \Sigma .
$$

Another notion of chaos was created in 1956 by Kac \cite{Kac} while he was studying the Boltzmann equation. 
For a fixed positive integer $n$, the Boltzmann equation describes the density function of the distribution of the velocities
of $n$ many molecules of dilute monoatomic gases where the energy is assumed to be transferred only via 
elastic collisions between the molecules. While the Boltzmann equation is a non-linear equation,
Kac came up with a linear integro-differential equation that he called the ``master equation" \cite[Equation~(2.6)]{Kac}. 
Kac's master equation depends on a positive integer $n$ and its solution has $n+1$ real variables 
$(x_1, \ldots , x_n, t)$. The $n$ real variables $(x_1, \ldots , x_n)$ belong on the ``Kac's sphere" $\K^n$
which stands for the sphere in $\R^n$ centered at the origin whose radius is equal to $\sqrt{n}$, (i.e. $\K^n = \{ 
(x_1, \ldots , x_n) \in \R^n : x_1^2+ \cdots + x_n^2=n \} $),
while the extra variable $t$ represents time. Kac seeked solutions $f^{(n)}$ to his master equation that are symmetric in the variables
$(x_1, \ldots , x_n)$ for every $t \geq 0$, i.e.
\begin{equation} \label{symmetric}
f^{(n)}(x_1 \ldots , x_n ,t)= f^{(n)}(x_{\pi(1)}, \ldots , x_{\pi (n)},t), \quad \text{for every permutation }\pi \text{ of } \{ 1, \ldots , n \} .
\end{equation}

Moreover for any set $E$, a function $g:E^n \rightarrow \mathbb{C}$ is called symmetric if 
\begin{equation}\label{sym-function}
g^\pi(x_1,x_2,...,x_n)=g(x_1,x_2,...,x_n)
\end{equation} for all permutations $\pi$ of $\{1,\ldots,n\}$ and for all $(x_1,x_2,...,x_n) \in E^n$, where for each permutation $\pi$ of $\{1,\ldots,n\}$ we define $g^\pi:E^n \rightarrow \mathbb{C}$ by
\begin{equation}\label{function-perm}
g^\pi(x_1,x_2,...,x_n):=g(x_{\pi(1)},x_{\pi(2)},...,x_{\pi(n)}).
\end{equation}

We also assume that the solution $f^{(n)}$ to Kac's master equation is a {\em density function on $\mathbb{K}^n$} i.e. we assume that $f^{(n)}d\sigma^n$ is a probability measure on the Borel subsets of $\mathbb{K}^n$ where $\sigma^n$ denotes the normalized uniform measure on $\mathbb{K}^n$. For each $1 \leq m \leq n$, we can define a probability measure $(f^{(n)}d\sigma^n)_m$ on the Borel subsets of $\mathbb{R}^m$ by
\begin{eqnarray*}
(f^{(n)}d\sigma^n)_m(A)=\int_{P_m^{-1}(A)}f^{(n)}d\sigma^n
\end{eqnarray*}
where $P_m:\mathbb{K}^n \rightarrow \mathbb{R}^m$ is the canonical projection into the first $m$ copies of $\mathbb{R}$. It is clear that $(f^{(n)}d\sigma^n)_m$ is absolutely continuous with respect to the $m$-dimensional Lebesgue measure $\lambda^m$ on $\mathbb{R}^m$, and thus by the Radon-Nikodym Theorem there exists a function $f_m^{(n)} \in L^1(\mathbb{R}^m)$, called the $m^{\text{th}}$ marginal function of $f^{(n)}$, such that
\begin{eqnarray*}
\int_A f_m^{(n)} d\lambda^m = (f^{(n)}d\sigma^n)_m(A)=\int_{P_m^{-1}(A)}f^{(n)}d\sigma^n
\end{eqnarray*}
for every Borel subset $A$ of $\mathbb{R}^m$. Hence $f^{(n)}_m d\lambda^m$ is a Borel probability measure on $\mathbb{R}^m$ for every $1 \leq m \leq n$. Kac observed that if $f^{(n)}$ satisfies the master equation \cite[Equation~(2.6)]{Kac} then the first and
second marginals $f_1^{(n)}$ and $f_2^{(n)}$ satisfy \cite[Equation~(3.7)]{Kac} which reads
\begin{equation} \label{3.7}
\frac{\partial f_1^{(n)}(x,t)}{\partial t } = \frac{(n-1)\nu}{2 \pi n} \int_{-\sqrt{n-x^2}}^{\sqrt{n-x^2}}
 f_2^{(n)} (x \cos \theta + y \sin \theta , - x \sin \theta + y \cos \theta , t) - f_2^{(n)} (x,y,t)  d \theta d y .
\end{equation}

This equality is interpreted in the weak sense for density functions, i.e. each side is integrated against smooth functions with compact support having variables $(x,t) \in \mathbb{R} \times [0,\infty)$. In particular, the derivative is interpreted in the sense of distributions.

Kac observed that if the limits $f_1(\cdot , t):= \lim_{n \to \infty}\limits f^{(n)}_1(\cdot ,t)$ and 
$f_2(\cdot , \cdot , t):= \lim_{n \to \infty}\limits f^{(n)}_2(\cdot , \cdot ,t)$ 
exist in $L^1(\R)$ and $L^1(\R^2)$ respectively,  for all $t \geq 0$, 
(where the dots denote arbitrary real variables, and the $L^1$ spaces are taken with respect to the Lebesgue measure),
 and if for almost all $x,y \in \R$,
\begin{equation} \label{3.8}
f_2 (x,y,t) = f_1 (x,t) f_1 (y,t)  ,
\end{equation}
then 
\begin{equation}\label{3.5}
\frac{\partial f_1 (x,t)}{\partial t} = \frac{\nu}{2 \pi} \int_{-\infty}^\infty \int_0 ^{2\pi} f_1(x \cos \theta + y \sin \theta , t)
f_1(-x \sin \theta + y \cos \theta , t) - f_1(x,t) f_1(y,t)  d \theta d y 
\end{equation}
again, in the weak sense, i.e. the function $f_1$ is a solution to a simplified version of the non-linear Boltzmann equation.

Equation~(\ref{3.8})  motivated Kac to introduce the following definition:
For all $n \in \N$ let $f^{(n)}$ be a symmetric, (as in (\ref{symmetric})), probability  density function defined on $\K ^n$ (i.e. $f^{(n)}d\sigma^n$ is a Borel probability measure on $\mathbb{K}^n$). 
For $1 \leq k \leq n$ let $f^{(n)}_k$ denote the $k^{\text{th}}$ marginal of $f^{(n)}$. 
The sequence $(f^{(n)})$ is said to have the ``Boltzmann property'' if for all $k \in \N$ the  limit $\lim_{n \to \infty}\limits f^{(n)}_k$ exists in $L^1(\R^k )$, and
moreover, if $f_1$ denotes the $L^1(\R )$ limit of $f^{(n)}_1$,  then for all  $k \in \N$ and for almost all $x_1, \ldots, x_k \in \R$:
\begin{equation} \label{Boltzmann}
\lim_{n \to \infty} f^{(n)}_k(x_1, \ldots , x_k)= \prod_{i=1}^k f_1(x_i) .
\end{equation}

 Kac proved that if the initial value solution (at time $t=0$) to the master equation is symmetric (as in (\ref{symmetric})) and satisfies the Boltzmann property
then the solution to the master equation is symmetric and
satisfies the Boltzmann property for all times $t >0$. Since Kac's master equation is linear hence the existence of its solution is guaranteed by 
well known theory, Kac produced a method for constructing a solution to a  simplified version of the non-linear Boltzmann equation.

The ``Boltzmann property'' is commonly referred to as ``Kac's chaos'' and has attracted the interest of many people such as 
McKean \cite{McKean}, Johnson \cite{Johnson}, Tanaka \cite{Tanaka}, Ueno \cite{Ueno}, Gr\"{u}nbaum \cite{Grunbaum}, 
Murata \cite{Murata},      Graham and M\'{e}l\'{e}ard \cite{GrahamMeleard},
 Sznitman \cite{Sznitman84}, \cite{Sznitman06}, Mischler \cite{Mischler}, Carlen, Carvalho and Loss \cite{CarlenCarvalhoLoss},
 Michler and Mouhot \cite{MischlerMouhot}. 
These authors considered a  more general situation 
than a sequence $f^{(n)}$ of density functions defined on $\K ^n$. They considered a topological space $E$ and a symmetric probability measure 
$\mu_n$ on the Borel $\sigma$-algebra $\mathcal{B}(E^n)$ of the Cartesian product $E^n$ for each $n \in \mathbb{N}$. Here are the relevant definitions,
where for a topological space $E$, we denote by $C_b(E)$ the set of all continuous bounded functions on $E$:
\begin{Def} \label{Def:symmetric}
Let $E$ be a topological space, $n$ be a positive integer, $\mu_n$ be a probability measure on the Borel subsets of $E^n$.
Then $\mu_n$ is called symmetric if for any $\phi_1,\phi_2,...,\phi_n \in C_b(E)$,
\begin{eqnarray*}
\int_{E^n} \phi_1(x_1)\phi_2(x_2) \cdots \phi_n(x_n) d \mu_n =\int_{E^n} \phi_1(x_{\pi(1)})\phi_2(x_{\pi(2)}) \cdots \phi_n (x_{\pi (n)}) d \mu_n
\end{eqnarray*}
for any permutation $\pi$ of $\{ 1, \ldots , n \}$.
\end{Def}

Note that if $f^{(n)}$ is a symmetric (in the sense of (\ref{symmetric})) density function on the Kac's sphere $\mathbb{K}^n$, $\sigma^n$ denotes, as above, the uniform Borel probability measure on $\mathbb{K}^n$, and $\widetilde{\sigma^n}$ denotes the extension of $\sigma^n$ to the Borel subsets of $\mathbb{R}^n$ such that the support of $\widetilde{\sigma^n}$ is equal to $\mathbb{K}^n$ (this is possible since $\mathbb{K}^n$ is a Borel subset of $\mathbb{R}^n$), 
then $f^{(n)} d \widetilde{\sigma^n}$ is a symmetric probability measure on $\R^n$, (in the sense of Definition~\ref{Def:symmetric}).
Thus Definition~\ref{Def:symmetric} gives a more general notion of symmetry than that of Equation~(\ref{symmetric}) that was considered by Kac.
Now we define the Boltzmann property, or Kac's chaos, but following the above mentioned literature, we use a more descriptive terminology:
\begin{Def} \label{Def:chaos}
Let $E$ be a topological space, $\nu$ be a Borel probability measure on $E$, and for every $n \in \mathbb{N}$ let $\mu_n$ be a symmetric (as in Definition \ref{Def:symmetric}) Borel probability measure on $E^n$. 
We say that $(\mu_n)_{n=1}^\infty$ is $\nu$-chaotic if for all $k \geq 1$ and  $\phi_1,\phi_2,...,\phi_k \in C_b(E)$,
$$
\lim\limits_{n \rightarrow \infty} \int_{E^n} \phi_1 ( x_1) \phi_2 (x_2) \cdots  \phi_k (x_k) d\mu_n = \prod\limits_{j=1}^k \int_E \phi_j (x) d \nu (x).
$$
\end{Def}
Now let $f^{(n)}$ be a symmetric (in the sense of (\ref{symmetric})) density on $\K ^n$ 
 for all $n \in \N$ such that the sequence $(f^{(n)})_n$ has the Boltzmann property (as defined by Kac). 
 In particular, let  $f_1$ be the $L^1(\R)$ limit of the sequence $(f^{(n)}_1)_n$. Extend each $f^{(n)}$ to $\R^n$ (without changing its name) by setting it equal to 
zero on $\R ^n \backslash \K ^n$, and let $\widetilde{\sigma^n}$ be the Borel probability measure on $\R^n$ which is supported 
on $\K ^n$  and it is  uniform on $\K ^n$. Then the sequence of measures $(f^{(n)} d \widetilde{\sigma^n})_n $ defined on the Borel subsets 
of $\R ^n$ is $\nu$-chaotic
where $ d \nu = f_1 dx $,  and $dx$ is the Lebesgue measure on $\R$. Thus Definition~\ref{Def:chaos} gives a more general 
notion of chaoticity than the Boltzmann property defined by Kac.

In this article we provide a relationship between the notion of mixing that appears in ergodic theory and Kac's chaos. Our main result is Theorem~\ref{Thm:main} which asserts that given a dynamical system 
on a separable metric space $E$ which satisfies a property similar to the mixing property and it is asymptotically stationary with stationary 
limit $\nu$, one can construct a sequence of Borel probability measures $(\mu_n)_n$ on $(E^n)_n$ which is $\nu$-chaotic.

Two other related forms of chaoticity that exist in literature are the chaoticity in the sense of Boltzmann entropy and the 
chaoticity in the sense of Fisher information. These two notions were introduced by Carlen, Carvalho, Le Roux, Loss, and Villani \cite{CCLRLV}.
Hauray and Mischler has shown that chaoticity in the sense of Fisher information implies chaoticity in the sense of Boltzmann entropy, 
which in turn implies Kac's chaoticity \cite[Theorem~1.4]{HaurayMischler}. Carrapatoso \cite{Carrapatoso} has extended the results of \cite{HaurayMischler}
to probability measures with support on the Boltzmann spheres. 

Finally we would like to mention that there is a vast literature on the 
notion of ``quantum chaos", where notions of ergodic theory are extended to quantum physical models. Without attempting to give detailed references to quantum chaos, we refer the interested reader to the books \cite{Casati}, \cite{Giannoni}, and \cite{Heiss} where some of these notions are presented.

\section{The statement of the main result and some examples}

In this section we state the main result of the article and we give several examples of dynamical systems that satisfy its assumptions. Before stating the main result we introduce some notation. If $E$ is a topological space then $\cB (E)$ will denote the $\sigma$-algebra of the Borel subsets of $E$, and $M(E)$ will denote the 
set of probability measures on $\cB (E)$. Also $\Sigma_n$ will denote the set of permutations of $\{1,\ldots,n\}$ for each $n \in \N$.

We now present the main result of the article.

\begin{Thm} \label{Thm:main}
Let $E$ be a separable metric space, $\mu$ be a probability measure on $\cB(E)$, and  $S:E \rightarrow E$ be a Borel measurable map. 
Assume that
\begin{enumerate}
\item[1.] For every $A \in \cB(E)$,
$$
\sup_{i \in \mathbb{N}} |\mu(S^{-i}A \cap S^{-k}S^{-i}A) - \mu(S^{-i}A)\mu(S^{-k}S^{-i}A)| \xrightarrow[k \rightarrow \infty]{} 0,
$$
and
\item[2.] $(E,\cB(E),\mu,S)$ is asymptotically stationary with stationary limit $\nu$.
\end{enumerate}
For every $n \in \N$ define 
 $\mu_n:\cB(E^n) \rightarrow [0,1]$ by 
 $$
 \mu_n(A) = \dfrac{1}{n!}\sum\limits_{\sigma \in \Sigma_n}\mu\{x\in E:(S^{\sigma(1)}(x),...,S^{\sigma(n)}(x)) \in A\}.
 $$ 
Then $(\mu_n)_{n \in \mathbb{N}}$ is $\nu$-chaotic.
\end{Thm}

Note that assumption 1 of the main result is related to the mixing property that was introduced in Definition~\ref{Def:ergodicity}(ii).
The differences between the two properties are: The limit in Definition~\ref{Def:ergodicity}(ii) is taken for any Borel sets $A$ and $B$ 
while in assumption~1, the sets $A$ and $B$ are equal and they belong to the $\sigma$-algebra $\{S^{-i}A: A \in \cB(E)\}$. In that sense, assumption 1 is weaker than 
Definition~\ref{Def:ergodicity}(ii). On the other hand, Definition~\ref{Def:ergodicity}(ii) lacks the uniformity which is 
manifested in assumption~1 by the presence of the supremum. In that sense, assumption 1 is stronger than 
Definition~\ref{Def:ergodicity}(ii).

One way to guarantee that a dynamical system satisfies assumption $1$ of Theorem~\ref{Thm:main} is by means of the next lemma.

\begin{Lem} \label{pi-system}
Let $(\Omega,\Sigma,\mu,S)$ be a dynamical system which is asymptotically stationary. Let $\Pi \subset \Sigma$ be a $\pi$-system such that $\sigma(\Pi)=\Sigma$ and

\begin{eqnarray}\label{mixing-pi-system}
\lim\limits_{k \rightarrow \infty}\sup\limits_{i \in \mathbb{N}} |\mu(S^{-i}A \cap S^{-k}S^{-i}B) - \mu(S^{-i}A)\mu(S^{-k}S^{-i}B)| = 0
\end{eqnarray}
for all $A, B \in \Pi$. Then (\ref{mixing-pi-system}) is satisfied for all $A,B\in\Sigma$ (hence assumption 1 of Theorem~\ref{Thm:main} is satisfied as well).
\end{Lem}

\begin{proof}
The proof will be by way of the Dynkin $\pi-\lambda$ Theorem. Fix $A \in \Pi$, and define
$\Lambda_A:=\{B \in \Sigma: \lim\limits_{k \rightarrow \infty}\sup\limits_{i \in \mathbb{N}} 
|\mu(S^{-i}A \cap S^{-k}S^{-i}B) - \mu(S^{-i}A)\mu(S^{-k}S^{-i}B)| = 0\}$. It is clear that $\Omega \in \Lambda_A$. Now, assume $B_1,B_2 \in \Lambda_A$ such that $B_1 \subset B_2$. Then we have
\begin{eqnarray*}
0 &\leq& \sup\limits_{i \in \mathbb{N}} |\mu(S^{-i}A \cap S^{-k}S^{-i}(B_2\setminus B_1)) - \mu(S^{-i}A)\mu(S^{-k}S^{-i}(B_2 \setminus B_1))|\\
&=& \sup\limits_{i \in \mathbb{N}} |\mu(S^{-i}A \cap S^{-k}S^{-i}B_2) - \mu(S^{-i}A \cap S^{-k}S^{-i}B_1)\\
&-&\mu(S^{-i}A)\mu(S^{-k}S^{-i}B_2) + \mu(S^{-i}A)\mu(S^{-k}S^{-i}B_1)|\\
&\leq& \sup\limits_{i \in \mathbb{N}} |\mu(S^{-i}A \cap S^{-k}S^{-i}B_2) - \mu(S^{-i}A)\mu(S^{-k}S^{-i}B_2)|\\
&+&\sup\limits_{i \in \mathbb{N}} |\mu(S^{-i}A \cap S^{-k}S^{-i}B_1)-\mu(S^{-i}A)\mu(S^{-k}S^{-i}B_1)|
\end{eqnarray*}
Taking limits on both sides as $k\rightarrow \infty$, we get that $B_2 \setminus B_1 \in \Lambda_A$. 

Lastly, let $(B_n)_{n=1}^\infty \subset \Lambda_A$ be any monotone increasing sequence with limit $B \in \Sigma$. We need to show that $B \in \Lambda_A$. Let $\epsilon > 0$. Denote by $\nu$ the stationary limit of $(\Omega,\Sigma,\mu,S)$. We have that

\begin{eqnarray}
\nonumber \hskip.3in &&\sup\limits_{i \in \mathbb{N}} |\mu(S^{-i}A \cap S^{-k}S^{-i}B) - \mu(S^{-i}A)\mu(S^{-k}S^{-i}B)|\\
\label{9}&\leq& \sup\limits_{i \in \mathbb{N}} |\mu(S^{-i}A \cap S^{-k}S^{-i}B) - \mu(S^{-i}A)\nu(S^{-i}B)| + \sup\limits_{i \in \mathbb{N}}|\mu(S^{-i}A)||\nu(B)-\mu(S^{-k}S^{-i}B)|
\end{eqnarray}

By assumption, there exists a $k_0 \in \mathbb{N}$ such that $|\nu(B)-\mu(S^{-k}B)| < \epsilon$ for all $k \geq k_0$, and thus, the second term of (\ref{9}) can be made small. Notice that there exists an $n_0 \in \mathbb{N}$ such that $|\nu(B_n)-\nu(B)| < \epsilon$ for all $n \geq n_0$. Using this information, we focus on the first term of line (\ref{9}),

\begin{eqnarray}
\nonumber \hskip.3in &&\sup\limits_{i \in \mathbb{N}}|\mu(S^{-i}A \cap S^{-k}S^{-i}B) - \mu(S^{-i}A)\nu(S^{-i}B)|\\
\nonumber &=&\sup\limits_{i \in \mathbb{N}} |\mu(S^{-i}A \cap S^{-k}S^{-i}B_{n_0}) + \mu(S^{-i}A \cap S^{-k}S^{-i}(B\setminus B_{n_0})) - \mu(S^{-i}A)\nu(S^{-i}B)|\\
 \label{11}&\leq& \sup\limits_{i \in \mathbb{N}} |\mu(S^{-i}A \cap S^{-k}S^{-i}B_{n_0}) - \mu(S^{-i}A)\nu(S^{-i}B)|+\sup\limits_{i \in \mathbb{N}} |\mu(S^{-i}A \cap S^{-k}S^{-i}(B\setminus B_{n_0}))|
\end{eqnarray}

The first term of line (\ref{11}) can be estimated as

\begin{eqnarray*}
&&\sup\limits_{i \in \mathbb{N}} |\mu(S^{-i}A \cap S^{-k}S^{-i}B_{n_0}) - \mu(S^{-i}A)\nu(S^{-i}B)|\\ 
&\leq& \sup\limits_{i \in \mathbb{N}} |\mu(S^{-i}A \cap S^{-k}S^{-i}B_{n_0}) - \mu(S^{-i}A)\mu(S^{-k}S^{-i}B_{n_0})| + \sup\limits_{i \in \mathbb{N}}|\mu(S^{-i}A)||\mu(S^{-k}S^{-i}B_{n_0})-\nu(B_{n_0})|\\
&+& \sup\limits_{i \in \mathbb{N}} |\mu(S^{-i}A)||\nu(B_{n_0})-\nu(B)|
\end{eqnarray*}

There exists a $k_1 \in \mathbb{N}$ such that $|\mu(S^{-k}S^{-i}B_{n_0})-\nu(B_{n_0})| < \epsilon$ for all $k \geq k_1$, and there exists a $k_2 \in \mathbb{N}$ such that $\sup\limits_{i \in \mathbb{N}} |\mu(S^{-i}A \cap S^{-k}S^{-i}B_{n_0}) - \mu(S^{-i}A)\mu(S^{-k}S^{-i}B_{n_0})|<\epsilon$ for all $k \geq k_2$. Thus, for all $k \geq \max\{k_1,k_2\}$, the first term of line (\ref{11}) is small.

The second term of line (\ref{11}) can be estimated as

\begin{eqnarray*}
&&\sup\limits_{i \in \mathbb{N}} |\mu(S^{-i}A \cap S^{-k}S^{-i}(B\setminus B_{n_0}))| \leq \sup\limits_{i \in \mathbb{N}} |\mu(S^{-k}S^{-i}(B\setminus B_{n_0}))|\\
&\leq& \sup\limits_{i \in \mathbb{N}}|\mu(S^{-k}S^{-i}B)-\nu(B)|+|\nu(B)-\nu(B_{n_0})|+\sup\limits_{i \in \mathbb{N}}|\nu(B_{n_0})-\mu(S^{-k}S^{-i}B_{n_0})|
\end{eqnarray*}
which is small for all $k \geq \max\{k_0,k_1\}$. Hence, we have that

$$\sup\limits_{i \in \mathbb{N}} |\mu(S^{-i}A \cap S^{-k}S^{-i}B) - \mu(S^{-i}A)\mu(S^{-k}S^{-i}B)| < 6\epsilon \text{ for all } k \geq \max\{k_0,k_1,k_2\}.$$
Thus $B \in \Lambda_A$. By the Dynkin $\pi-\lambda$ Theorem, for all $A \in \Pi$ and all $B \in \Sigma$ we have $\lim\limits_{k \rightarrow \infty}\sup\limits_{i \in \mathbb{N}} |\mu(S^{-i}A \cap S^{-k}S^{-i}B) - \mu(S^{-i}A)\mu(S^{-k}S^{-i}B)| = 0$. The same argument can be turned around to show that for a fixed $B \in \Sigma$ (\ref{mixing-pi-system}) holds for all $A \in \Sigma$.
\end{proof}

Using Lemma~\ref{pi-system}, we obtain the following corollary of Theorem~\ref{Thm:main}.

\begin{Cor}
Let $E$ be a separable metric space, $\mu$ be a probability measure on $\cB(E)$, and  $S:E \rightarrow E$ be Borel measurable. Let $\Pi \subset \Sigma$ be a $\pi$-system such that $\sigma(\Pi)=\Sigma$.
Assume that
\begin{enumerate}
\item[1.] For every $A,B \in \Pi$,
$$
\sup_{i \in \mathbb{N}} |\mu(S^{-i}A \cap S^{-k}S^{-i}B) - \mu(S^{-i}A)\mu(S^{-k}S^{-i}B)| \xrightarrow[k \rightarrow \infty]{} 0,
$$
and
\item[2.] $(E,\cB(E),\mu,S)$ is asymptotically stationary with stationary limit $\nu$.
\end{enumerate}
For every $n \in \N$ define 
 $\mu_n:\cB(E^n) \rightarrow [0,1]$ by 
 $$\mu_n(A) = \dfrac{1}{n!}\sum\limits_{\sigma \in \Sigma_n}\mu\{x\in E:(S^{\sigma(1)}(x),...,S^{\sigma(n)}(x)) \in A\}.$$
Then $(\mu_n)_{n \in \mathbb{N}}$ is $\nu$-chaotic.
\end{Cor}

The next two remarks give sufficient conditions for the assumptions of Theorem~\ref{Thm:main} to be met.

\begin{Rmk} \label{Rmk:simple}
Let $(\Omega , \Sigma, \mu , S)$ be a dynamical system which is mixing and $S$ is $\mu$-measure preserving. 
Then the assumptions 1 and 2 of Theorem~\ref{Thm:main} are satisfied for this dynamical system.
\end{Rmk}

Indeed, for every $A \in \Sigma$ we have 

\begin{eqnarray*}
\sup\limits_{i \in \N}|\mu(S^{-i}A \cap S^{-k}S^{-i}A) - \mu(S^{-i}A)\mu(S^{-k}S^{-i}A)|
&=& \sup\limits_{i \in \N}|\mu(S^{-i}(A \cap S^{-k}A))-\mu(A)\mu(S^{-k}A)|\\
&=&|\mu(A \cap S^{-k}A)-(\mu(A))^2|\xrightarrow[k \rightarrow \infty]{}0
\end{eqnarray*}
where the second equality is valid because $S$ is $\mu$-measure preserving and the limit is valid because the dynamical system is mixing. Thus assumption $1$ of Theorem~\ref{Thm:main} is satisfied. Also, $(\Omega,\Sigma,\mu,S)$ is asymptotically stationary with stationary limit $\mu$ since $S$ is $\mu$-measure preserving. Thus assumption $2$ of Theorem~\ref{Thm:main} is satisfied as well.

\begin{Rmk} \label{Rmk: pi-system}
Let $(\Omega,\Sigma,\mu,S)$ be a dynamical system and let $\Pi \subset \Sigma$ be a $\pi$-system such that $\sigma(\Pi)=\Sigma$.
\begin{enumerate}
\item[(i)] If $\lim\limits_{k \rightarrow \infty} |\mu(S^{-k}A \cap B)-\mu(S^{-k}A)\mu(B)|=0$ holds for every $A,B \in \Pi$, then it holds for every $A,B\in\Sigma$.
\item[(ii)] If $\mu(S^{-1}(A))=\mu(A)$ holds for all $A \in \Pi$ then it holds for all $A \in \Sigma$.
\end{enumerate}
\end{Rmk}
See Shalizi and Kontorovich \cite[Theorem $384$]{Shalizi} for the proof of part $(i)$. The proof of part $(ii)$ is very easy using Dynkin's $\pi-\lambda$ theorem.

We now present three examples of dynamical systems that satisfy the assumptions of Remark~\ref{Rmk:simple}. The first example is called 
the ``baker's map". The measure space for the baker's map is $([0,1]^2 := [0,1]\times[0,1],\cB ([0,1]^2),\mu)$ where $\mu$ is the Lebesgue measure restricted to
$\cB([0,1]^2)$. The map $S:[0,1]^2 \rightarrow [0,1]^2 $ of the dynamical system is defined by
$$
S(x,y)=
\begin{cases}
(2x,\frac{1}{2}y) & 0 \leq x < \frac{1}{2},0\leq y \leq 1\\
(2x-1,\frac{1}{2}y+\frac{1}{2}) & \frac{1}{2} \leq x \leq 1, 0 \leq y \leq 1
\end{cases} .
$$
Lasota and Mackey \cite[Example~4.3.1]{Lasota} prove that  the baker's map is mixing by verifying 
Definition~\ref{Def:ergodicity}(ii)  for all rectangles $A$, $B$ with sides parallel to $x$ and $y$ axes.  These rectangles form a $\pi$-system that generates
the $\sigma$-algebra $\cB ( [0,1]\times [0,1]) $. Thus by Remark~\ref{Rmk: pi-system}(i) the baker's map is mixing. Using the same $\pi$-system and Remark~\ref{Rmk: pi-system}(ii) it is easy to verify that the baker's map is measure preserving.

Another example of a dynamical system which satisfies the assumptions of Remark~\ref{Rmk:simple} is the Anosov map, (also called the cat map).
The measure space for the Anosov map is $([0,1)^2:= [0,1)\times[0,1),\cB( [0,1)^2 ),\mu)$ where $\mu$ is the Lebesgue measure restricted to 
$\cB ([0,1)^2 )$. 
The map $S:[0,1)^2 \rightarrow [0,1)^2 $ for the Anosov map is defined by 
$$
S(x,y)=(x+y,x+2y)\,  (\text{mod }1). 
$$
Lasota and Mackey prove that the cat map is mixing by using the Fibonacci sequence and Fourier transforms \cite[Example~4.4.3]{Lasota}. 
Arnold and Avez show that the Anosov Map is measure preserving \cite[Example~1.16]{ArnoldAvez}.

We now present a construction of an infinite product of probability measures satisfying the assumptions of Remark~\ref{Rmk:simple}. 
If $(\Omega_n, \Sigma_n , \mu_n)_{n \in \N}$ is a sequence of probability spaces then the Cartesian product 
$\prod_{n=1}^\infty  \Omega_n$ can be naturally equipped with an infinite product of these measures,
as defined by Kakutani \cite{Kakutani}. Denote this infinite product probability space by 
$(\prod_{n=1}^\infty \Omega_n , \prod_{n=1}^\infty \Sigma_n, \prod_{n=1}^\infty \mu_n )$. 
The $\sigma$-algebra $\prod_{n=1}^\infty \Sigma_n $ is generated by the $\pi$-system $\prod_{n=1}^{<\infty} \Sigma_n$ consisting of all sets of the form 
$\prod_{n=1}^\infty A_n$ where $A_n \in \Sigma_n$ for all $n \in \N$ and $A_n=\Omega_n$ for all but finitely many $n$'s. If $\Sigma_n=\Sigma$ for all $n \in \mathbb{N}$ then let $\Sigma^{<\infty}$ denote the $\pi$-system $\prod_{n=1}^{<\infty} \Sigma$. 
If $A= \prod_{n=1}^\infty A_n$ is such a set then we define $(\prod_{n=1}^\infty \mu_n) (A)= \prod_{n=1}^\infty(\mu_n (A_n))$. 
If  $\Omega_n=\Omega$ and $\Sigma_n=\Sigma$ for all $n \in \N$ then $\prod_{n=1}^\infty \Omega_n$ is denoted by $\Omega^\N$, 
and $\prod_{n=1}^\infty \Sigma_n$ is denoted by $\Sigma^\N$. If moreover 
$\mu_n = \mu$ for all $n \in \N$ then $\prod_{n=1}^\infty \mu_n$ is denoted by $\mu^\N$. 
Assume that for every $n \in \mathbb{N}$ $(\Omega , \Sigma , \mu_n , S)$ is a dynamical system (i.e. in general we may allow different measures to be considered on the same $\sigma$-algebra $\Sigma$). Define $S^\N : \Omega^\N \to \Omega^\N$ by 
$$
S^\N ((\omega_n)_{n \in \N}) = (S(\omega_{n+1}))_{n \in \N} .
$$
Then $S^\N$ is measurable i.e. $(\Omega^\N,\Sigma^\N,\prod_{n=1}^\infty \mu_n,S^\N)$ is a dynamical system. Indeed, it is enough and easy to check that $(S^\N)^{-1}(A) \in \Sigma^\N$ for every set $A$ in the $\pi$-system $\Sigma^{<\infty}$. Obviously, if $(\Omega,\Sigma,\mu,S)$ is a dynamical system and $S$ is $\mu$-measure preserving, then $S^\N$ is $\mu^\N$-measure preserving, (it is enough to be verified on sets of 
the $\pi$-system $\prod_{n=1}^{<\infty}\Sigma_n$, which is an easy task). Thus, if $(\Omega, \Sigma , \mu, S)$ denotes either
the baker's dynamical system, or the Anosov dynamical system defined above, then $S^\N $ is $\mu^N$- measure preserving. 
We claim that for any dynamical system $(\Omega , \Sigma , \mu , S)$, $(\Omega^\N, \Sigma^\N , \mu^\N , S^\N )$
is always mixing. By Remark~\ref{Rmk: pi-system}(i) this claim is enough and easy to be verified for sets $A$, $B$ in the $\pi$-system $\Sigma^{< \infty}$.
Indeed if $A= \prod_{n=1}^\infty A_n$, $B= \prod_{n=1}^\infty B_n$ where $A_n, B_n \in \Sigma$ for all $n$ and 
$A_n=B_n=\Omega$ for all $n > m$, then 
$$
\mu^\N ((S^\N)^{-k}A )= \prod_{n=1}^m \limits \mu (S^{-k} A_n ) , \quad \mu^{\N} (B)= \prod_{i=1}^m\limits  \mu (B_i) ,
$$
and 
$$
\mu^{\N}  ((S^{\N})^{-k} A \cap B ) = \prod_{i=1}^m\limits  \mu (B_i) \prod_{n=1}^m \limits \mu(S^{-k} A_n) \quad \text{for all }k >m. 
$$
Hence
$$
| \mu^\N ((S^\N)^{-k}A \cap B) - \mu^\N ((S^\N)^{-k} A) \mu^\N (B) |=0 \quad \text{for all }k>m.
$$
Thus if $(\Omega , \Sigma , \mu , S)$ is a dynamical system such that $S$ is $\mu$-measure preserving, then we obtain that
$(\Omega^\N , \Sigma^\N , \mu^\N , S^\N )$ is a dynamical system that satisfies the assumptions of Remark~\ref{Rmk:simple}. 

In all the examples that we have mentioned so far, the map of the dynamical system is measure preserving.  
Such maps are trivially asymptotically stationary with stationary limits being equal to the original measure. 
We now describe how infinite product probability measures can be used to give examples of dynamical systems that 
satisfy the assumptions of Theorem~\ref{Thm:main} without the map of the dynamical system being measure preserving. 
In this example the stationary limit of the dynamical system is different than the original measure.
Let $\Omega = [0,1]$ and $\Sigma = \cB ([0,1])$. For every $k \in \N$ define a density function 
$\phi_k:[0,1] \rightarrow \{ 0, 1,  2 \}$ by 
$$
\phi_k (x)=\chi_{[0,1-\frac{1}{2^{k-1}}]}(x)+2\chi_{(1-\frac{1}{2^{k-1}},1-\frac{1}{2^k}]}(x)
$$ 
and  the probability measure $\mu_k : \Sigma \rightarrow [0,1]$ by 
$$
\mu_k(A):=\int_A \phi_k (x) dx \quad \text{ for all }A \in \Sigma .
$$
Consider the probability space $(\Omega^\N,\Sigma^\N,\prod_{k=1}^\infty \mu_k)$ and in order to make the notation easier let $\mathbb{M}=\prod_{k=1}^\infty \mu_k$ and $\mathbb{L}=\lambda^\N$ where $\lambda$ is the Lebesgue measure on $[0,1]$. Consider a map $S:\Omega \rightarrow \Omega$ such that $S^\N$ is $\mathbb{L}$-measure preserving but not $\mathbb{M}$-measure preserving (this is valid for example when $S$ is the identity map). In order to make the notation easier let $\mathbb{S}=S^\N$. We claim that 
the dynamical system $(\Omega^\N , \Sigma^\N , \mathbb{M} , \mathbb{S})$ satisfies the assumptions of Theorem~\ref{Thm:main}. Indeed, in order to verify assumption 2 of Theorem~\ref{Thm:main}, we prove that $\mathbb{M}$ is asymptotically stationary with stationary limit 
equal to $\mathbb{L}$. 

Fix any $A \in \Sigma^\N$. Define for each $k \in \mathbb{N}$ the set $C_k:=[0,1]^k \times [0,1-\frac{1}{2^k}] \times [0,1-\frac{1}{2^{k+1}}] \times [0,1-\frac{1}{2^{k+2}}] \times \cdots$. Since $C_k = \bigcap\limits_{n=0}^\infty \left([0,1]^k \times [0,1-\frac{1}{2^k}] \times \cdots \times [0,1-\frac{1}{2^{k+n}}] \times [0,1]^\N\right)$, we have that $C_k \in \Sigma^\N$, and if $B \in \Sigma^\N$ with $B \subset C_k$ then $\mathbb{M}(B)=\mathbb{L}(B)$. Thus for any $A \in \Sigma^\N$ we have that

\begin{eqnarray*}
&&|\mathbb{M}(\mathbb{S}^{-k}A)-\mathbb{L}(A)|=|\mathbb{M}(\mathbb{S}^{-k}A)-\mathbb{L}(\mathbb{S}^{-k}B)|\\
&=&|\mathbb{M}((\mathbb{S}^{-k}A) \cap C_k)+\mathbb{M}((\mathbb{S}^{-k}A) \setminus C_k) - \mathbb{L}((\mathbb{S}^{-k}A) \cap C_k)-\mathbb{L}((\mathbb{S}^{-k}A) \setminus C_k)|\\
&=&|\mathbb{M}((\mathbb{S}^{-k}A) \setminus C_k)-\mathbb{L}((\mathbb{S}^{-k}A) \setminus C_k)| \leq 2\left(1-\prod\limits_{s=k}^\infty (1-\frac{1}{2^s})\right) \xrightarrow[k \rightarrow \infty]{} 0,
\end{eqnarray*}
where the last inequality is valid because $\mathbb{M}(C_k)=\mathbb{L}(C_k)=\prod\limits_{s=k}^\infty (1-\frac{1}{2^s})$, hence $\mathbb{M}([0,1]^\N \setminus C_k)=\mathbb{L}([0,1]^\N \setminus C_k) = 1-\prod\limits_{s=k}^\infty (1-\frac{1}{2^s})$. This verifies assumption $2$ of Theorem~\ref{Thm:main}.

Now, in order to verify assumption $1$ of Theorem~\ref{Thm:main} we use Lemma~\ref{pi-system}. Fix sets $A,B \in \prod\limits_{n=1}^{<\infty} \cB([0,1])$. Then $A=\prod\limits_{n=1}^\infty A_n$ where $A_n \in \cB([0,1])$ and there exists $N \in \N$ such that $A_n=[0,1]$ for all $n >N$. Then by the definition of the infinite product measure we have that for all $k > N$
$$
\mathbb{M}((\mathbb{S}^{-i}A) \cap (\mathbb{S}^{-k}\mathbb{S}^{-i}B)) = 
\mathbb{M}(\mathbb{S}^{-i}A) \mathbb{M}(\mathbb{S}^{-k}\mathbb{S}^{-i}B) \text{ for all } i \in \N .
$$
Hence (\ref{mixing-pi-system}) is valid, and the assumptions of Lemma~\ref{pi-system} are met. This means assumption $1$ of Theorem~\ref{Thm:main} is valid.

\section{Proof of the Main Result}

Many times when deciding whether a sequence of measures is $\nu$-chaotic, it is easier to show one of the equivalent formulations of chaos. Sznitman proves various equivalences to the definition of chaos which we list below.

\begin{Thm}\cite[Proposition $2.2$]{Sznitman06}\label{Equiv}
Let $E$ be a separable metric space, $(\mu_n)_{n=1}^\infty$ a sequence of symmetric probability measures on $E^n$ (as in Definition~\ref{Def:symmetric}), and $\nu$ be a probability measure on $E$. The following are equivalent:
\begin{enumerate}
\item[1.] The sequence $(\mu_n)_{n=1}^\infty$ is $\nu$-chaotic (as in Definition~\ref{Def:chaos}).
\item[2.] The function $X_n:E^n \rightarrow M(E)$ defined by $X_n:=\dfrac{1}{n}\sum\limits_{i=1}^n \delta_{x_i}$ (where $\delta_x$ stands for the Dirac measure at $x$) converges in law with respect to $\mu_n$ to the constant random variable $\nu$, i.e. for every $g \in C_b(E)$ we have that
$$\int_{E^n}|(X_n-\nu)g|^2 d\mu_n \xrightarrow[n \rightarrow \infty]{}0,$$
where $C_b(E)$ stands for the space of bounded continuous scalar valued functions on $E$.
\item[3.] The sequence $(\mu_n)_{n=1}^\infty$ satisfies Definition~\ref{Def:chaos} with $k=2$.
\end{enumerate}
\end{Thm}

In order to construct examples of sequences of symmetric probability measures satisfying condition 2 of Theorem~\ref{Equiv} we will show that it is sufficient to construct a sequence of probability measures (not necessarily symmetric) which satisfy the same condition. Given a measurable space $(E,\Sigma)$, $n \in \mathbb{N}$, and a probability measure $\mu_n$ on the product space $(E^n,\Sigma^n)$, we define a symmetric probability measure $\mu_n^{\text{sym}}$ on $(E^n,\Sigma^n)$ in the following way: For each $\sigma \in \Sigma_n$ define $\Pi_\sigma:E^n \rightarrow E^n$ by $\Pi_\sigma(x_1,...,x_n)=(x_{\sigma(1)},...,x_{\sigma(n)})$, and define the probability measure $\mu_n^\sigma:\Sigma^n \rightarrow [0,1]$ by $\mu_n^\sigma(A)=\mu_n(\Pi_\sigma(A))$ for each $A \in \Sigma^n$. It is easy to verify that for any bounded and measurable function $f:E^n \rightarrow \mathbb{C}$ we have that 
$$\int_{E^n} f d\mu_n^\sigma=\int_{E^n} f^\sigma d\mu_n$$ where $f^\sigma$ is defined as in (\ref{function-perm}). The symmetric probability measure $\mu_n^{\text{sym}}:\Sigma^n\rightarrow [0,1]$ is then defined by $$\mu_n^{\text{sym}}(A)=\dfrac{1}{n!}\sum\limits_{\sigma \in \Sigma_n} \mu_n^\sigma(A) \text{ for each } A \in \Sigma^n.$$ For any bounded and measurable function $f:E^n \rightarrow \mathbb{C}$ which is symmetric (as in (\ref{sym-function})),
\begin{eqnarray}\label{int-equiv}
\int_{E^n} f d\mu_n^\text{sym} = \dfrac{1}{n!} \sum\limits_{\sigma \in \Sigma_n} \int_{E^n} f d\mu_n^\sigma = \dfrac{1}{n!}\sum\limits_{\sigma \in \Sigma_n} \int_{E^n} f^\sigma d\mu_n = \int_{E^n}f d\mu_n.
\end{eqnarray}
For a fixed $g \in C_b(E)$ and $\nu \in M(E)$ if we apply (\ref{int-equiv}) for $f:=|(X_n-\nu)(g)|^2:E^n \rightarrow \mathbb{C}$ (which is obviously bounded, measurable, and symmetric), we obtain the following.
\begin{Rmk}\label{sym}
Let $E$ be a separable metric space, $\mu_n$ be a probability measure on $\cB(E^n)$, and $\nu$ be a probability measure on $\cB(E)$. For any fixed $g \in C_b(E)$ we have that
\begin{eqnarray*}
\int_{E^n}|(X_n-\nu)g|^2 d\mu_n \xrightarrow[n \rightarrow \infty]{} 0 \,\,\, \text{ if and only if } \,\,\,\int_{E^n}|(X_n-\nu)g|^2 d\mu_n^\text{sym} \xrightarrow[n \rightarrow \infty]{} 0.
\end{eqnarray*}
\end{Rmk}

In order to prove Theorem~\ref{Thm:main}, we will also need the following.

\begin{Prop}\label{Prop:main}
Let $E$ be a separable metric space, $\mu$ be a probability measure on $\cB(E)$, and  $S:E \rightarrow E$ be a Borel measurable map. 
Assume that
\begin{enumerate}
\item[1.] For every $A \in \cB(E)$,
$$
\sup_{i \in \mathbb{N}} |\mu(S^{-i}A \cap S^{-k}S^{-i}A) - \mu(S^{-i}A)\mu(S^{-k}S^{-i}A)| \xrightarrow[k \rightarrow \infty]{} 0,
$$
and
\item[2.] $(E,\cB(E),\mu,S)$ is asymptotically stationary with stationary limit $\nu$.
\end{enumerate}
For every $n \in \N$ define 
 $\mu_n:\cB(E^n) \rightarrow [0,1]$ by 
 $$
 \mu_n(A) = \mu\{x\in E:(S(x),...,S^n(x)) \in A\}.
 $$ 
Then for every $g \in C_b(E)$,
$$\int_{E^n}|(X_n-\nu)g|^2 d\mu_n \xrightarrow[n \rightarrow \infty]{} 0. $$
\end{Prop}

Notice that for the measure $\mu_n$ defined in Proposition~\ref{Prop:main}, and for each $A \in \cB(E^n)$,
\begin{eqnarray}\label{measure}
\nonumber \mu_n^{\text{sym}}(A)&=&\dfrac{1}{n!}\sum\limits_{\sigma \in \Sigma_n} \mu_n^\sigma(A)=\dfrac{1}{n!}\sum\limits_{\sigma \in \Sigma_n} \mu_n(\Pi_\sigma (A))\\
\nonumber&=&\dfrac{1}{n!}\sum\limits_{\sigma \in \Sigma_n} \mu\{x\in E:(S(x),...,S^n(x)) \in \Pi_\sigma(A)\}\\
\nonumber &=& \dfrac{1}{n!}\sum\limits_{\sigma \in \Sigma_n} \mu\{x\in E:(S^{\sigma^{-1}(1)}(x),...,S^{\sigma^{-1}(n)}(x)) \in A\}\\
&=& \dfrac{1}{n!}\sum\limits_{\sigma \in \Sigma_n} \mu\{x\in E:(S^{\sigma(1)}(x),...,S^{\sigma(n)}(x)) \in A\}.
\end{eqnarray} 
Now the proof of Theorem~\ref{Thm:main} follows immediately from Proposition~\ref{Prop:main}, Remark~\ref{sym}, the fact that $\mu_n^{\text{sym}}$ is a symmetric probability measure, Theorem~\ref{Equiv}, and (\ref{measure}). It remains to prove Proposition~\ref{Prop:main}.

\begin{proof}[Proof of Proposition~\ref{Prop:main}]

First, let $g:=\chi_{E_1}$ for some $E_1 \in \cB(E)$. Then, using that for $1 \leq i<j \leq n$, we have
 
 \begin{eqnarray*}
\int_{E^n} g(x_i)g(x_j) d\mu_n &=& \mu_n(E^{i-1} \times E_1 \times E^{j-i-1} \times
 E_1 \times E^{n-j}) = \mu(S^{-i}(E_1) \cap S^{-j}(E_1))
\end{eqnarray*}

we can write

\begin{eqnarray}
\nonumber\int_{E^n} \left|(X_n - \nu)(g)\right|^2 d\mu_n 
&=& \int_{E^n} \left|\dfrac{1}{n}\sum\limits_{i=1}^n g(x_i) - \int_E g d\nu\right|^2 d\mu_n \\
\nonumber&=&\dfrac{1}{n^2}\sum_{1\leq i,j \leq n} \int_{E^n} g(x_i)g(x_j)d\mu_n - \dfrac{2\nu(E_1)}{n} \sum\limits_{i=1}^n \int_{E^n} g(x_i) d\mu_n  + (\nu(E_1))^2\\
\nonumber&=&\dfrac{1}{n^2} \sum\limits_{1 \leq i,j \leq n} \mu(S^{-i}(E_1) \cap S^{-j}(E_1)) - \dfrac{2\nu(E_1)}{n} \sum\limits_{1 \leq i \leq n} \mu(S^{-i}(E_1))+(\nu(E_1))^2\\ 
&=&\dfrac{2}{n^2} \sum\limits_{1 \leq i < j \leq n} \mu(S^{-i}(E_1) \cap S^{-(j-i)}(S^{-i}(E_1))) \label{12} \\ &+& \dfrac{1}{n^2} \sum\limits_{1 \leq i \leq n} \mu(S^{-i}(E_1))
  - \dfrac{2 \nu(E_1)}{n} \sum\limits_{1 \leq i \leq n} \mu(S^{-i}(E_1)) + (\nu(E_1))^2 \nonumber
\end{eqnarray}

We have

\begin{eqnarray*}
\dfrac{1}{n^2}\sum\limits_{1 \leq i \leq n} \mu(S^{-i}(E_1)) &\leq& \dfrac{1}{n^2}\sum\limits_{1 \leq i \leq n} 1 \xrightarrow[n \rightarrow \infty]{} 0
\end{eqnarray*}

and by assumption $2$,

\begin{eqnarray*}
\dfrac{2 \nu(E_1)}{n} \sum\limits_{1 \leq i \leq n} \mu(S^{-i}(E_1)) 
\xrightarrow[n \rightarrow \infty]{} 2(\nu(E_1))^2
\end{eqnarray*}

Also, line (\ref{12}) can be written as

\begin{eqnarray}
\nonumber &&\dfrac{2}{n^2} \sum\limits_{1 \leq i < j \leq n} \mu(S^{-i}(E_1) \cap S^{-(j-i)}(S^{-i}(E_1)))\\
&=& \dfrac{2}{n^2} \sum\limits_{i=1}^{n-1}\sum\limits_{k=1}^{n-i} \left[\mu\left(S^{-i}(E_1) \cap S^{-k}(S^{-i}(E_1))\right) - \mu(S^{-i}(E_1))\mu(S^{-k}(S^{-i}(E_1)))\right]\label{13}\\
&+& \dfrac{2}{n^2} \sum\limits_{i=1}^{n-1} \sum\limits_{k=1}^{n-i} \mu(S^{-i}(E_1))\mu(S^{-k}(S^{-i}(E_1))).\label{14}
\end{eqnarray}

First, let us focus on line (\ref{13}). Let $\epsilon > 0$. By assumption $1$ of Theorem~\ref{Thm:main},
 there exists $K_0 \in \mathbb{N}$ such that if $k \geq K_0$ then $$|\mu(S^{-i
 }(E_1)\cap S^{-k}(S^{-i}(E_1)))- \mu(S^{-i}(E_1))\mu(S^{-k}(S^{-i}(E_1)))| <
  \epsilon$$ for every $i$. Thus for $n > K_0 + 1$ we have that line (\ref{13}) is less than or equal to

\begin{eqnarray*}
&& \dfrac{2}{n^2}\sum\limits_{k=1}^{K_0} \sum\limits_{i=1}^{n-k} \left|\mu\left(S^{-i}(E_1) \cap S^{-k}(S^{-i}(E_1))\right) - \mu(S^{-i}(E_1))\mu(S^{-k}(S^{-i}(E_1)))\right| + \dfrac{2}{n^2} \sum\limits_{k=K_0+1}^{n-1} \sum\limits_{i=1}^{n-k} \epsilon.
\end{eqnarray*}

Since the first double sum has at most $K_0^2+\frac{K_0 n}{2}$ terms, the second double sum has at most $\frac{(n-K_0)n}{2}$ terms, and $0 \leq \left|\mu\left(S^{-i}(E_1) \cap S^{-k}(S^{-i}(E_1))\right) - \mu(S^{-i}(E_1))\mu(S^{-k}(S^{-i}(E_1)))\right| \leq 2$, we have that line (\ref{13}) is less than or equal to

\begin{eqnarray*}
\dfrac{4(K_0^2+K_0 n/2)}{n^2} &+& \dfrac{\epsilon(n-K_0)n}{n^2} \xrightarrow[n \rightarrow \infty]{} \epsilon.
 \end{eqnarray*}

Now we will focus on line (\ref{14}). By assumption $2$, there exists $N_0 \in \mathbb{N}$ such that if 
$n \geq N_0$ then

\begin{eqnarray*}
|\mu(S^{-n}(E_1)-\nu(E_1)|<\epsilon .
\end{eqnarray*}
Hence, line (\ref{14}) is equal to

\begin{eqnarray*}
&&\dfrac{2}{n^2} \sum\limits_{i=1}^{N_0} \sum\limits_{k=1}^{n-i} \mu(S^{-i}(E_1))\mu(S^{-k}(S^{-i}(E_1)))+\dfrac{2}{n^2} \sum\limits_{i=N_0+1}^{n-1} \sum\limits_{k=1}^{n-i} \mu(S^{-i}(E_1))\mu(S^{-k}(S^{-i}(E_1)))
\end{eqnarray*}

The first double sum has at most $N_0^2+\frac{N_0 n}{2}$ terms, and thus,

\begin{eqnarray*}
\dfrac{2}{n^2} \sum\limits_{i=1}^{N_0} \sum\limits_{k=1}^{n-i} \mu(S^{-i}(E_1))\mu(S^{-k}(S^{-i}(E_1))) \leq \dfrac{2(N_0^2+N_0n/2)}{n^2}  \xrightarrow[n \rightarrow \infty]{} 0
\end{eqnarray*}

The second double sum can be rewritten as

\begin{eqnarray*}
&&\dfrac{2}{n^2} \sum\limits_{i=N_0+1}^{n-1} \sum\limits_{k=1}^{n-i} \mu(S^{-i}(E_1))\mu(S^{-k}(S^{-i}(E_1)))\\
&=& \dfrac{2}{n^2}\dfrac{(n-1-N_0)^2}{2}(\nu(E_1))^2\\
&+& \dfrac{2}{n^2} \sum\limits_{i=N_0+1}^{n-1} \sum\limits_{k=1}^{n-i} \left[\mu(S^{-i}E_1)[\mu(S^{-k}S^{-i}E_1)-\nu(E_1)]+[\mu(S^{-i}E_1)-\nu(E_1)]\nu(E_1)\right].
\end{eqnarray*}
We have
\begin{eqnarray*}
\dfrac{2}{n^2}\dfrac{(n-1-N_0)^2}{2}(\nu(E_1))^2 \xrightarrow[n \rightarrow \infty]{} (\nu(E_1))^2
\end{eqnarray*}
 and
 \begin{eqnarray*}
&&\dfrac{2}{n^2} \sum\limits_{i=N_0+1}^{n-1} \sum\limits_{k=1}^{n-i} \left[\mu(S^{-i}E_1)[\mu(S^{-k}S^{-i}E_1)-\nu(E_1)]+[\mu(S^{-i}E_1)-\nu(E_1)]\nu(E_1)\right]\\
&\leq& \dfrac{2}{n^2}\dfrac{(n-1)^2}{2}\epsilon + \dfrac{2}{n^2}\dfrac{(n-1)^2}{2}\epsilon \xrightarrow[n \rightarrow \infty]{} (\nu(E_1))^2 + 2\epsilon.
 \end{eqnarray*}
Hence line $(\ref{14})$ converges to $(\nu(E_1))^2$ as $n \rightarrow \infty$. This shows

\begin{eqnarray*}
\int_{E^n} \left|(X_n - \nu)(g)\right|^2 d\mu_n
 \xrightarrow[n \rightarrow \infty]{} 0
 \end{eqnarray*}
  for $g=\chi_{E_1}$.

Now, consider the simple function $g := \sum\limits_{k=1}^K \alpha_k \chi_{E_k}$. We have

\begin{eqnarray*}
&&\left(\int_{E^n} \left| (X_n-\nu)\left(g\right)\right
|^2d\mu_n\right)^{1/2}=\left(\int_{E^n} \left| (X_n-\nu)\left(\sum\limits_{k=1}^K \alpha_k \chi_{E_k}\right)\right
|^2d\mu_n\right)^{1/2}\\
&=&\left(\int_{E^n} \left| \sum\limits_{k=1}^K
 \alpha_k\left[(X_n-\nu)\chi_{E_k}\right]\right|^2d\mu_n\right)^{1/2} \leq
  \sum\limits_{k=1}^K |\alpha_k| \left(\int_{E^n} \left
  |(X_n-\nu)\chi_{E_k}\right|^2d\mu_n\right)^{1/2}
   \xrightarrow[n \rightarrow \infty]{} 0
   \end{eqnarray*}
since the limit is zero for each characteristic function and we have a finite sum.

Finally, let $g \in C_b(E)$ and $\epsilon > 0$. There exists a simple function $G$ such that $|g(x) - G(x)| < \epsilon$ for all $x \in E$. We have that

\begin{eqnarray*}
\left(\int_{E^n}\left|(X_n-\nu)(g-G)\right|^2d\mu_n\right)^{1/2} &\leq& \left(\int_{E^n} |X_n(g-G)|^2d\mu_n\right)^{1/2} + \left(\int_{E^n} |\nu(g-G)|^2d\mu_n\right)^{1/2}\\
&=&\left(\int_{E^n} \left|\dfrac{1}{n} \sum\limits_{i=1}^n(g-G)(x_i)\right|^2d\mu_n\right)^{1/2}+|\nu(g-G)|\\
&\leq& \dfrac{1}{n} \sum\limits_{i=1}^n \left(\int_{E^n} |(g-G)(x_i)|^2 d\mu_n\right)^{1/2} + \left|\int_E (g-G)d\nu\right|\\
&\leq& \sum\limits_{i=1}^n\dfrac{\epsilon}{n} + \epsilon = 2\epsilon
\end{eqnarray*}

and since $G$ is a simple function,

\begin{eqnarray*}
\left(\int_{E^n} |(X_n-\nu)G|^2d\mu_n\right)^{1/2} 
\xrightarrow[n \rightarrow \infty]{} 0.
\end{eqnarray*}

Therefore by the triangle inequality on $L^2(E^n,d\mu_n)$ norm we obtain since $\epsilon$ is arbitrary,

\begin{eqnarray*}
\left(\int_{E^n} |(X_n-\nu)g|^2d\mu_n\right)^{1/2}
 \xrightarrow[n\rightarrow \infty]{} 0.
 \end{eqnarray*}

\end{proof}


\begin{thebibliography}{9}


\bibitem{ArnoldAvez}
V.I. Arnold, A. Avez, {\em Ergodic properties of classical mechanics}, Benjamin (1968).

\bibitem{Bogachev}
V. Bogachev, {\em Measure Theory} Volume 1, Springer Verlag, (2007).

\bibitem{CCLRLV}
E.A. Carlen, M.C. Carvalho, J. Le Roux, M. Loss, C. Villani, {\em Entropy and chaos in the Kac model}, 
Kinet. Relat. Models {\bf 3} (1) (2010), 85-122.

\bibitem{CarlenCarvalhoLoss}
E. Carlen, M.C. Carvalho, M. Loss, {\em Kinetic theory and the Kac master equation}, Entropy \& Quantum II,
Contemp. Math. {\bf 552}, (2011), 1-20.

\bibitem{Carrapatoso}
K. Carrapatoso, Quantitative and qualitative Kac's chaos on the Boltzmann's sphere,
Ann. Inst. H. Poincaré Probab. Statist.
{\bf 51}, Number 3 (2015), 993-1039.

\bibitem{Casati}
G. Casati, B. Chirikov (Editors), {\em Quantum Chaos, between order and disorder}, Cambridge University Press, (1995).

\bibitem{Giannoni}
M. Giannoni, A. Voros, J. Zinn-Justin (Editors), {\em Chaos and Quantum Physics (Les Houches)}, North Holland, (1991).

\bibitem{GrahamMeleard}
G. Graham, S. M\'{e}l\'{e}ard, {\em Stochastic particle approximations for generalized Boltzmann models and convergence estimates},
Ann. Propab. {\bf 25} (1997), 115-132.

\bibitem{Grunbaum}
F.A. Gr\"{u}nbaum, {\em Propagation of chaos for the Boltzmann equation}, Arch. Rational Mech. Anal. {\bf 42} (1971), 323-345.


\bibitem{HaurayMischler}
M. Hauray, S. Mischler, {On Kac's chaos and related problems}, J. Funct. Anal. {\bf 266} (2014), 6055-6157.

\bibitem{Heiss}
W. Dieter Heiss (Ed.) {\em Chaos and Quantum Chaos}, Proceedings of the Eighth Chris Engelbrecht Summer School in Theoretical Physics. Held at Blydepoort, Eastern Transvaal South Africa, 13-24 January 1992, Lecture Notes in Physics, Springer Verlag.

\bibitem{Johnson}
D.P. Jonhson, {\em On a class of stochastic processes and its relationship to infinite particle gases}, Trans. Amer. Math. Soc. {\bf 132} (1968),
275-295.

\bibitem{Kac}
M. Kac, {\em Foundations of Kinetic Theory}, Proceedings of the Third Berkeley Symposium on Mathematical 
Statistics and Probability: held at the Statistical Laboratory, University of California, (1956).

\bibitem{Kakutani}
S. Kakutani, {\em On Equivalence of Infinite Product Measures}, Annals of Mathematics, Vol. {\bf 49}, No. 1, (Jan. 1948), p.214-224.

\bibitem{Shalizi}
A. Kontorovich, C. Shalizi, {\em Almost None of the Theory of Stochastic Processes}, Version 0.1.1, July 3, 2010, available in author's website:
https://www.stat.cmu.edu/~cshalizi/almost-none/

\bibitem{Lasota}
A. Lasota, M.C. Mackey, {\em Probabilistic Properties of Deterministic Systems}, Cambridge University Press, (1985).


\bibitem{McKean}
H.P. McKean, {\em Propagation of chaos for a class of non-linear parabolic equations}, Stochastic Differential Equations,
(Lecture Series in Differential Equations Session 7, Catholic University, 1967), 41-57, Air Force Office Sci. Res. Arlington, Va.

\bibitem{Mischler}
S. Mischler, {\em Le programme de Kac sur les limites de champ moyen}, in: S\'{e}minaire EDP-X, D\'{e}cembre 2010.

\bibitem{MischlerMouhot}
S. Mischler, C. Mouhot, {\em Kac's program in kinetic theory}, Invent. Math. {\bf 193} (1) (2014) 1-147.

\bibitem{Murata}
H. Murata, {\em Propagation of chaos for Boltzmann like equations of non-cut off type in the plane}, Hiroshima Math. J. {\bf 7} (1977), 479-515.

\bibitem{Sinai}
Y. Sinai, {\em Ergodic Teory}, Acta Physica Austriaca, Suppl. X, (1973), pp. 575-608.

\bibitem{Sznitman84}
A-S. Sznitman, {\em \'{E}quations de type de Boltzmann spatialement homog\`{e}nes}, 
Z. Wahrscheinlichkeitstheorie verw. Gebiete {\bf 66} (1984), 559-592.

\bibitem{Sznitman06}
A-S. Sznitman, {\em Topics in Propagation of Chaos}, Lecture Notes in Mathematics (2006), pp. 165-251.


\bibitem{Tanaka}
H. Tanaka, {\em Propagation of chaos for certain Markov processes of jump type with non-linear generators},
Part I, Part II, Proc. Japan Acad. {\bf 45} (1969), 449-452, 598-599.


\bibitem{Ueno}
T. Ueno, {\em A class of Markov processes with non-linear bounded generators}, Japan J. Math. {\bf 38}, (1969), 19-38.



\end{thebibliography}
\end{document}